\pdfoutput=1
\RequirePackage{ifpdf}
\ifpdf 
\documentclass[pdftex]{sigma}
\else
\documentclass{sigma}
\fi

\numberwithin{equation}{section}
\newtheorem{Theorem}{Theorem}[section]
\newtheorem{Corollary}[Theorem]{Corollary}
\newtheorem{Lemma}[Theorem]{Lemma}
\newtheorem{Proposition}[Theorem]{Proposition}
 { \theoremstyle{definition}
\newtheorem{Definition}[Theorem]{Definition}

\newtheorem{Remark}[Theorem]{Remark} }

\def\dd{{\rm d}}

\begin{document}

\allowdisplaybreaks

\newcommand{\arXivNumber}{1606.03948}

\renewcommand{\PaperNumber}{061}

\FirstPageHeading

\ShortArticleName{An Energy Gap for Complex Yang--Mills Equations}

\ArticleName{An Energy Gap for Complex Yang--Mills Equations}

\Author{Teng HUANG~$^{\dag\ddag}$}

\AuthorNameForHeading{T.~Huang}

\Address{$^\dag$~Key Laboratory of Wu Wen-Tsun Mathematics, Chinese Academy of Sciences, P.R.~China}
\Address{$^\ddag$~School of Mathematical Sciences, University of Science and Technology of China, P.R.~China}
\EmailD{\href{mailto:oula143@mail.ustc.edu.cn}{oula143@mail.ustc.edu.cn}}

\ArticleDates{Received May 31, 2017, in f\/inal form July 26, 2017; Published online August 08, 2017}

\Abstract{We use the energy gap result of pure Yang--Mills equation
[Feehan P.M.N., \textit{Adv. Math.} \textbf{312} (2017), 547--587] to prove another energy gap result of complex Yang--Mills equations
[Gagliardo M., Uhlenbeck K., \textit{J.~Fixed Point Theory Appl.} \textbf{11} (2012), 185--198], when Riemannian manifold $X$ 
of dimension $n\geq 2$ satisf\/ies certain conditions.}


\Classification{58E15; 81T13}

\section{Introduction}

Let $X$ be an oriented $n$-manifold endowed with a smooth Riemannian metric~$g$. Let $P$ be a~principle $G$-bundle over~$X$. The structure group $G$ is assumed to be a compact Lie group with Lie algebra $\mathfrak{g}$. We denote a connection on $P$ as $A$ and its curvature as $F_{A}$. Let $\mathfrak{g}_{P}$ be the adjoint bundle of $P$. We def\/ine by $\dd_{A}$ the exterior covariant derivative on section of $\Lambda^{\bullet}T^{\ast}X\otimes(P\times_{G}\mathfrak{g}_{P})$ with respect to the connection $A$. The curvature $\mathcal{F}_{\mathcal{A}}$ of the complex connection $\mathcal{A}:=\dd_{A}+\sqrt{-1}\phi$, {$\phi\in\Omega^{1}(X,\mathfrak{g}_{P})$} is a two-form with values in $P\times_{G}(\mathfrak{g}_{P}^{{\mathbb C}})$:
\begin{gather*}\mathcal{F}_{\mathcal{A}}=\big[\big(\dd_{A}+\sqrt{-1}\phi\big)\wedge\big(\dd_{A}+\sqrt{-1}\phi\big)\big]
=F_{A}-\tfrac{1}{2}[\phi\wedge\phi]+\sqrt{-1}\dd_{A}\phi.\end{gather*}
The complex Yang--Mills functional is def\/ined in any dimension as the norm squared of the complex curvature \cite[Section~3]{Gagliardo/Uhlenbeck:2012}
\begin{gather*}{\rm YM}_{{\mathbb C}}(A,\phi):=\int_{X}|\mathcal{F}_{\mathcal{A}}|^{2}=\int_{X}\big(|F_{A}-\phi\wedge\phi|^{2}+|\dd_{A}\phi|^{2}\big).\end{gather*}
This functional reduces to the pure Yang--Mills functional when the extra f\/ield $\phi$ vanishes. The Euler--Lagrange equations for this functional are
\begin{gather}
\dd^{\ast}_{A}(F_{A}-\phi\wedge\phi)+(-1)^{n}\ast[\phi,\ast \dd_{A}\phi]=0,\nonumber\\
\dd_{A}^{\ast}\dd_{A}\phi-(-1)^{n}\ast[\phi,\ast(F_{A}-\phi\wedge\phi)]=0.\label{E1.1}
\end{gather}
They can also be succinctly written as
\begin{gather*}\dd^{\ast}_{\mathcal{A}}\mathcal{F}_{\mathcal{A}}=0.\end{gather*}
These equations are not elliptic, even after the real gauge-equivalence, so it is necessary to add the moment map condition \begin{gather*}\dd_{A}^{\ast}\phi=0.\end{gather*}
In this article, we call $\dd_A+ \sqrt{-1}\phi$ the solution of the complex Yang--Mills equations, it not only satisf\/ies equations~(\ref{E1.1}), but also satisf\/ies the moment map condition. These equations called \emph{complex} just because the connection $\dd_A+\sqrt{-1}\phi$ is a $1$-form with value in $P\times_{G}(\mathfrak{g}_{\mathbb{C}})$, the base manifold is always a real Riemannian manifold.
\begin{Remark}
 Some examples of complex Yang--Mills equations:
 \begin{enumerate}\itemsep=0pt
 \item[(1)] Kapustin--Witten equations \cite{Gagliardo/Uhlenbeck:2012}
\begin{gather*}
(F_{A}-\phi\wedge\phi)^{+}=0,\qquad (\dd_{A}\phi)^{-}=0,\qquad \dd_{A}\ast\phi=0
\end{gather*}
on four-dimensional manifolds.
\item[(2)] On a stable Higgs bundle $(E,\theta)$, there exists a Hermitian metric $h$ such that the Hitchin--Simpson connection $\dd_{A_{h}}+\theta+\theta^{\ast,h}$ satisf\/ies the Einstein condition \cite{Hitchin,Simpson1988}
\begin{gather*}\sqrt{-1}\Lambda_{\omega}\big(F_{A_{h}}+\big[\theta,\theta^{\ast,h}\big]\big)-\lambda \operatorname{Id}_{E}=0,\end{gather*}
the connection $\dd_{A_{h}}+\theta+\theta^{\ast,h}$ also satisf\/ies complex Yang--Mills equations.
\end{enumerate}
\end{Remark}
In particular, it is easy to see that if $\phi=0$, the complex Yang--Mills equations will reduce to the pure Yang--Mills equation $\dd_{A}^{\ast}F_{A}=0$. Many researchers have studied the energy gap of Yang--Mills equation. The motivation of these gap results is partly from physics and partly from math which would be help to better understand the behavior of the Yang--Mills functional near its critical points. In~\cite{Bourguignon/Lawson,Bourguignon/Lawson/Simons,Dodziuk/Min-Oo,Gerhardt,Min-Oo}, they all require some positive hypotheses on the curvature tensors of a Riemannian metric. In~\cite{Feehan2015}, Feehan applied the Lojasiewicz--Simon gradient inequa\-li\-ty~\cite[Theorem~3.2]{Feehan2015} to remove the positive hypothesis on the Riemannian curvature tensors. In~\cite{Huang2017}, the author gave another proof of energy gap theorem of pure Yang--Mills equation without using Lojasiewicz--Simon gradient inequality. We also want to understand the behaviour of the \emph{complex Yang--Mills} functional near its critical points. It is an interesting question to consider, whether the complex Yang--Mills equations have the energy gap phenomenon. In this article, we give a positive answer to this question when $X$ satisf\/ies certain conditions.
\begin{Theorem}\label{T1.1}
Let $X$ be a closed, oriented, smooth Riemannian manifold of dimension $n\geq2$ with smooth Riemannan metric $g$, $P$ be a $G$-bundle over $X$ with $G$ being a compact Lie group, let $2p>n$ when $n\neq 2,4$ or $p\geq2$ when $n=2,4$. Then there is a positive constant $\varepsilon=\varepsilon(n,p,g)$ with the following significance. Suppose that all flat connections on $P$ are non-degenerate in the sense of Definition~{\rm \ref{D3.6}}. If the pair $(A,\phi)$ is a~$C^{\infty}$-solution of complex Yang--Mills equations over $X$, the curvature $F_{A}$ of connection $A$ obeys
\begin{gather*}
\|F_{A}\|_{L^{p}(X)}\leq\varepsilon,
\end{gather*}
then $A$ is a flat connection and $\phi$ must vanish.
\end{Theorem}
\begin{Remark}
In fact, the extra f\/ields will vanish if the Ricci tensor of a Riemannian metric of~$X$ is positive (see Corollary~\ref{C2.3}). For a general Riemannian metric, we will not know whether $\ker{\Delta_{\Gamma}}|_{\Omega^{1}(X,\mathfrak{g}_{P})}=\{0\}$ ($\Gamma$ is any f\/lat connection on~$P$) unless we assume some topological hypotheses for $X$, such as $\pi_{1}(X)=\{1\}$, so $P\cong X\times G$ if and only if $P$ is f\/lat \cite[Theorem~2.2.1]{Donaldson/Kronheimer}. In this case, $\Gamma$ is gauge-equivalent to the product connection and $\ker{\Delta_{\Gamma}}|_{\Omega^{1}(X,\mathfrak{g}_{P})}\cong H^{1}(X,\mathbb{R})$, so the hypothesis for $X$ ensures that the kernel vanishes.
\end{Remark}
If $\mathcal{F}_{\mathcal{A}}=0$, i.e., $F_{A}-\phi\wedge\phi=0$ and $\dd_{A}\phi=0$, then we call $\mathcal{A}=\dd_A+\sqrt{-1}\phi$ a complex f\/lat connection. Now, we denote the moduli space of solutions of complex f\/lat-connections by
\begin{gather*}\mathcal{M}(P,g):=\big\{(A,\phi)\,|\, (F_{A}-\phi\wedge\phi)=0\ \text{and}\ \dd_{A}\phi=\dd^{\ast}_{A}\phi=0\big\}/\mathcal{G}_{P}.\end{gather*}
In particular, the moduli space $M(P,g)$ can be embedded into $\mathcal{M}(P,g)$ via $A\mapsto(A,0)$, where
\begin{gather*}M(P,g):=\{\Gamma\colon F_{\Gamma}=0\}/\mathcal{G}_{P},\end{gather*}
is the moduli space of gauge-equivalence class $[\Gamma]$ of f\/lat connection $\Gamma$ on $P$. Obviously, the complex f\/lat connection also satisf\/ies the complex Yang--Mills equations. Using the gap Theorem~\ref{T1.1} of complex Yang--Mills equations, we can have a gap result for the extra f\/ields as follows
\begin{Corollary}\label{C1.2}
Let $X$ be a closed, oriented, smooth Riemannian manifold of dimension $n\geq2$ with smooth Riemannan metric~$g$, $P$ be a $G$-bundle over $X$ with $G$ being a compact Lie group. Then there is a positive constant $\varepsilon=\varepsilon(g,n)$ with the following significance. Suppose that all flat connections on $P$ are non-degenerate in the sense of Definition~{\rm \ref{D3.6}}. If the pair $(A,\phi)$ is a~$C^{\infty}$-solution of complex flat connection over~$X$, the $L^{2}$-norm of extra field $\phi$ obeys
\begin{gather*}
\|\phi\|_{L^{2}(X)}\leq\varepsilon,
\end{gather*}
then $\phi$ vanishes and $A$ is a flat connection. In particular, if $M(P,g)$ and $\mathcal{M}(P,g)\backslash M(P,g)$ are both not empty, then the moduli space $\mathcal{M}(P,g)$ is non-connected.
\end{Corollary}
The organization of this paper is as follows. In Section~\ref{section2}, f\/irst we set our notations and recall some basic def\/initions in dif\/ferential geometry. Next, we recall some identities and some estimates for the solutions of complex Yang--Mills equations which were proved by Gagliardo and Uhlenbeck~\cite{Gagliardo/Uhlenbeck:2012}. Finally, we recall an energy gap result of the pure Yang--Mills equation due to Feehan~\cite{Feehan2015}. Since the Theorem~\ref{T2.4} plays an essential role in our proof of our main result, we provide more details to prove the theorem by another method. In Section~\ref{section3}, we def\/ine the least eigenvalue $\lambda(A)$ of $\dd_{A}\dd^{\ast}_{A}+\dd^{\ast}_{A}\dd_{A}|_{\Omega^{1}(X,\mathfrak{g}_{P})}$ with respect to connection~$A$. We extend the idea of Feehan~\cite{Feehan2014.09} to prove that $\lambda(A)$ has a lower bound that is uniform with respect to~$[A]$ obeying $\|F_{A}\|_{L^{p}(X)}\leq\varepsilon$ for a small enough $\varepsilon=\varepsilon(g,n,p)$ under some conditions for $g$, $G$, $P$, and~$X$. We conclude Section~\ref{section4} with the proofs of Theorem~\ref{T1.1} and Corollary~\ref{C1.2}.

\section{Fundamental preliminaries}\label{section2}
We shall generally adhere to the now standard gauge-theory conventions and notation of Donaldson and Kronheimer \cite{Donaldson/Kronheimer} and Feehan \cite{Feehan2015}. Throughout our article, $G$ denotes a compact Lie group and $P$ a smooth principal $G$-bundle over a compact Riemannnian manifold $X$ of dimension $n\geq2$ and endowed with Riemannian metric $g$, $\mathfrak{g}_{P}$ denote the adjoint bundle of $P$, endowed with a $G$-invariant inner product and $\Omega^{p}(X,\mathfrak{g}_{P})$ denote the smooth $p$-forms with va\-lues in~$\mathfrak{g}_{P}$. Given a~connection on~$P$, we denote by $\nabla_{A}$ the corresponding covariant derivative on $\Omega^{\ast}(X,\mathfrak{g}_{P})$ induced by $A$ and the Levi-Civita connection of~$X$. Let $\dd_{A}$ denote the exterior derivative associated to~$\nabla_{A}$.

For $u\in L^{p}(X,\mathfrak{g}_{P})$, where $1\leq p<\infty$ and $k$ is an integer, we denote
\begin{gather*}
\|u\|_{L^{p}_{k,A}(X)}:=\left(\sum_{j=0}^{k}\int_{X}|\nabla^{j}_{A}u|^{p}{\rm dvol}_{g}\right)^{1/p},
\end{gather*}
where $\nabla^{j}_{A}:=\nabla_{A}\circ\cdots\circ\nabla_{A}$ (repeated $j$ times for $j\geq0$). For $p=\infty$, we denote
\begin{gather*}
\|u\|_{L^{\infty}_{k,A}(X)}:=\sum_{j=0}^{k}\operatorname{ess}\sup_{X}\big|\nabla^{j}_{A}u\big|.
\end{gather*}
\subsection{Identities for the solutions}\label{section2.1}
In this section, we recall some basic identities that the solutions to complex Yang--Mills connections obey. A nice discussion of these identities can be found in~\cite{Gagliardo/Uhlenbeck:2012}.
\begin{Theorem}[Weitezenb\"{o}ck formula]
\begin{gather}\label{40}
\dd_{A}^{\ast}\dd_{A}+\dd_{A}\dd_{A}^{\ast}=\nabla^{\ast}_{A}\nabla_{A}+\operatorname{Ric}(\cdot)+\ast[\ast F_{A},\cdot]\qquad \text{on}\quad \Omega^{1}(X,\mathfrak{g}_{P}),
\end{gather}
where ${\rm Ric}$ is the Ricci tensor.
\end{Theorem}
\begin{Proposition}[{\cite[Theorem 4.3]{Gagliardo/Uhlenbeck:2012}}]\label{F3}
If $\dd_{A}+\sqrt{-1}\phi$ is a solution of the complex Yang--Mills equations, then
\begin{gather}\label{Y1}
\nabla^{\ast}_{A}\nabla_{A}\phi+\operatorname{Ric}\circ\phi+\ast[\ast(\phi\wedge\phi),\phi]=0.
\end{gather}
\end{Proposition}
By integrating \eqref{Y1} over $X$, we have an identity
\begin{gather}\label{E2.3}
\|\nabla_{A}\phi\|^{2}_{L^{2}(X)}+\langle \operatorname{Ric}\circ\phi,\phi\rangle_{L^{2}(X)}+2\|\phi\wedge\phi\|^{2}_{L^{2}(X)}=0.
\end{gather}
Then the results by Gagliardo--Uhlenbeck give a following consequence result for the extra f\/ields.
\begin{Corollary}[{\cite[Corollary 4.5]{Gagliardo/Uhlenbeck:2012}}]\label{C2.3}
If $X$ is a compact manifold with a positive Ricci curvature, then solutions of the complex Yang--Mills equations reduce to solutions of the pure Yang--Mills equation with $\phi=0$.
\end{Corollary}
\begin{Proposition}[energy identity]
If $\dd_{A}+\sqrt{-1}\phi$ is a solution of the complex Yang--Mills equations, then
\begin{gather*}{\rm YM}_{\mathbb{C}}(A,\phi)=\|F_{A}\|^{2}_{L^{2}(X)}-\|\phi\wedge\phi\|^{2}_{L^{2}(X)}.\end{gather*}
\end{Proposition}
\begin{proof}
By using the moment condition $\dd_{A}^{\ast}\phi=0$, the complex Yang--Mills functional is written as
\begin{gather*}
{\rm YM}_{{\mathbb C}}(A,\phi)=\int_{X}\big(|F_{A}-\phi\wedge\phi|^{2}+|\dd_{A}\phi|^{2}+|\dd^{\ast}_{A}\phi|\big)\\
\hphantom{{\rm YM}_{{\mathbb C}}(A,\phi)}{} =\int_{X}|F_{A}|^{2}+|\phi\wedge\phi|^{2}-2\langle F_{A},\phi\wedge\phi\rangle+|\dd_{A}\phi|^{2}+|\dd_{A}^{\ast}\phi|\\
\hphantom{{\rm YM}_{{\mathbb C}}(A,\phi)}{}=\int_{X}|F_{A}|^{2}+|\phi\wedge\phi|^{2}+|\nabla_{A}\phi|^{2}+\langle \operatorname{Ric}\circ\phi,\phi\rangle\\
\hphantom{{\rm YM}_{{\mathbb C}}(A,\phi)}{} =\|F_{A}\|^{2}_{L^{2}(X)}-\|\phi\wedge\phi\|^{2}_{L^{2}(X)}.
\end{gather*}
For the last identity, we use the equation~(\ref{E2.3}).
\end{proof}

As an application of the maximum principle, Gagliardo--Uhlenbeck obtain a priori $L^{\infty}$-estimate for the extra f\/ields.
\begin{Theorem}[{\cite[Corollary~4.6]{Gagliardo/Uhlenbeck:2012}}]\label{F6}
Let $G$ be a compact Lie group, $P$ be a $G$-bundle over a~closed, smooth manifold $X$ of dimension $n\geq2$ and endowed with a smooth Riemannian metric~$g$. Then there is a positive constant $C=C(g,n)$ with the following significance. If~$(A,\phi)$ is a smooth solution of complex Yang--Mills equation, then
\begin{gather*}
\|\phi\|_{L^{\infty}(X)}\leq C\|\phi\|_{L^{2}(X)}.
\end{gather*}
\end{Theorem}

\subsection{Energy gap for Yang--Mills connections}\label{section2.2}
In this section, f\/irst we recall an energy gap result of Yang--Mills equation.
\begin{Theorem}[{\cite[Theorem~1.1]{Feehan2015}}]\label{T2.4}
Let $X$ be a closed, oriented, smooth Riemannian manifold of dimension $n\geq2$ with smooth Riemannan metric~$g$, $P$ be a $G$-bundle over $X$, let
{$2p\geq n$ when $n\geq3$ or $p\geq2$ when $n=2$}. Then any Yang--Mills connection~$A$ over~$X$ with compact Lie group~$G$ is either satisfies
\begin{gather*}{\int_{X}|F_{A}|^{p}{\rm dvol}_{g}\geq C_{0}}\end{gather*}
for a constant $C_{0}>0$ depending only on $X$, $n$, $p$, $G$ or the connection~$A$ is flat.
\end{Theorem}
In \cite{Huang2017}, the author proves the energy gap theorem of Yang--Mills connection without using the Lojasiewicz--Simon gradient inequality. Here, we give a~proof in detail for the readers' convenience. We review a key result due to Uhlenbeck for the connections with $L^{p}$-small curvature ($2p>n$)~\cite{Uhlenbeck1985} which provides existence of a f\/lat connection~$\Gamma$ on~$P$, a global gauge transformation $u$ of $A$ to Coulomb gauge with respect to~$\Gamma$ and a Sobolev norm estimate for the distance between~$\Gamma$ and~$A$.
\begin{Theorem}[{\cite[Corollary 4.3]{Uhlenbeck1985} and \cite[Theorem 5.1]{Feehan2015}}]\label{T2}
Let $X$ be a closed, smooth manifold of dimension $n\geq2$ and endowed with a Riemannian metric~$g$, and $G$ be a compact Lie group, and $2p>n$. Then there are constants, $\varepsilon=\varepsilon(n,g,G,p)\in(0,1]$ and $C=C(n,g,G,p)\in[1,\infty)$, with the following significance. Let $A$ be a $L^{p}_{1}$ connection on a principal $G$-bundle~$P$ over~$X$. If the curvature~$F_{A}$ obeys
\begin{gather*}\|F_{A}\|_{L^{p}(X)}\leq\varepsilon,\end{gather*}
then there exists a flat connection, $\Gamma$, on $P$ and a gauge transformation $u\in L^{p}_{2}(X)$ such that
\begin{enumerate}\itemsep=0pt
\item[$(1)$] $\dd^{\ast}_{\Gamma} ({u}^{\ast}(A)-\Gamma)=0$ on $X$,
\item[$(2)$] $\|u^{\ast}(A)-\Gamma\|_{L^{p}_{1,\Gamma}}\leq C\|F_{A}\|_{L^{p}(X)}$, and
\item[$(3)$] $\|u^{\ast}(A)-\Gamma\|_{L^{\frac{n}{2}}_{1,\Gamma}}\leq C\|F_{A}\|_{L^{\frac{n}{2}}(X)}$.
\end{enumerate}
\end{Theorem}
Next, we also review another key result due to Uhlenbeck \cite[Theorem~3.5]{Uhlenbeck} concerning a~priori estimate for the curvature of a~Yang--Mills connection over a~closed Riemannian manifold.
\begin{Theorem}[{\cite[Corollary 4.6]{Feehan2015}}]\label{T3}
Let $X$ be a compact manifold of dimension ${n\geq3}$ and endowed with a Riemannian metric~$g$, let $A$ be a smooth Yang--Mills connection with respect to the metric~$g$ on a smooth $G$-bundle $P$ over~$X$. Then there exist constants $\varepsilon=\varepsilon(n,g)>0$ and $C=C(n,g)$ with the following significance. If the curvature $F_{A}$ obeys
\begin{gather*}\|F_{A}\|_{L^{\frac{n}{2}}(X)}\leq\varepsilon,\end{gather*}
then
\begin{gather*}\|F_{A}\|_{L^{\infty}(X)}\leq C\|F_{A}\|_{L^{2}(X)}.\end{gather*}
\end{Theorem}

\begin{proof}[Proof Theorem \ref{T2.4}] For any $p\geq 2n$ ($n\geq3$), the estimate in Theorem \ref{T3} yields
\begin{gather}\label{E2.4}
\|F_{A}\|_{L^{p}(X)}\leq C\|F_{A}\|_{L^{\infty}(X)}\leq C\|F_{A}\|_{L^{2}(X)},
\end{gather}
for $C=C(g,n)$. If $n\geq4$, then
\begin{gather}\label{E2.5}
\|F_{A}\|_{L^{2}(X)}\leq C\|F_{A}\|_{L^{\frac{n}{2}}(X)}.
\end{gather}
If $n=3$, the $L^{p}$ interpolation \cite[equation~(7.9)]{Gilbarg/Trudinger} implies that
\begin{gather*}
\|F_{A}\|_{L^{2}(X)}\leq C\|F_{A}\|^{3/4}_{L^{\frac{3}{2}}(X)}\|F_{A}\|^{1/4}_{L^{\infty}(X)} \leq C\|F_{A}\|^{3/4}_{L^{\frac{3}{2}}(X)}\|F_{A}\|^{1/4}_{L^{2}(X)},
\end{gather*}
and thus
\begin{gather}\label{E2.6}
\|F_{A}\|_{L^{2}(X)}\leq C\|F_{A}\|_{L^{\frac{3}{2}}(X)}.
\end{gather}
Therefore, combining (\ref{E2.4})--(\ref{E2.6}), we obtain
\begin{gather*}\|F_{A}\|_{L^{p}(X)}\leq C\|F_{A}\|_{L^{n/2}(X)}, \qquad \forall\, 2p\geq n\quad \text{and}\quad n\geq3.\end{gather*}
Hence, if we suppose $\|F_{A}\|_{L^{s}(X)}$ suf\/f\/iciently small ($2s\geq n$ when $n\geq3$ or $s\geq 2$ when $n=2$) such that $\|F_{A}\|_{L^{q}(X)}$ ($2q>n$ when $n\geq3$ or $q\geq2$ when $n=2$) satisf\/ies the hypothesis of Theorem~\ref{T2}, then Theorem~\ref{T2} will provide a f\/lat connection~$\Gamma$ on~$P$, and a gauge transformation $u\in\mathcal{G}_{P}$ and the estimate
\begin{gather*}\|u^{\ast}(A)-\Gamma\|_{L^{q}_{1}(X)}\leq C(q)\|F_{A}\|_{L^{q}(X)},\end{gather*}
and
\begin{gather*}\dd_{\Gamma}^{\ast}(u^{\ast}(A)-\Gamma)=0.\end{gather*}
We denote $\tilde{A}:=u^{\ast}(A)$ and $a:=u^{\ast}(A)-\Gamma$, then the curvature of $\tilde{A}$ is
\begin{gather*}F_{\tilde{A}}=\dd_{\Gamma}a+a\wedge a.\end{gather*}
The connection $\tilde{A}$ also satisf\/ies Yang--Mills equation
\begin{gather}\label{G4}
0=\dd_{\tilde{A}}^{\ast}F_{\tilde{A}}.
\end{gather}
Hence taking the $L^{2}$-inner product of~(\ref{G4}) with~$a$, we obtain
\begin{gather*}
0=(\dd_{\tilde{A}}^{\ast}F_{\tilde{A}},a)_{L^{2}(X)} =(F_{\tilde{A}},\dd_{\tilde{A}}a)_{L^{2}(X)}\\
\hphantom{0}{} =(F_{\tilde{A}},\dd_{\Gamma}a+2a\wedge a)_{L^{2}(X)} =(F_{\tilde{A}},F_{\tilde{A}}+a\wedge a)_{L^{2}(X)}.
\end{gather*}
Then we get
\begin{gather*}
\|F_{A}\|^{2}_{L^{2}(X)}=\|F_{\tilde{A}}\|^{2}_{L^{2}(X)} =-(F_{\tilde{A}},a\wedge a)_{L^{2}(X)}\\
\hphantom{\|F_{A}\|^{2}_{L^{2}(X)}}{}
\leq\|F_{\tilde{A}}\|_{L^{2}(X)}\|a\wedge a\|_{L^{2}(X)} =\|F_{A}\|_{L^{2}(X)}\|a\wedge a\|_{L^{2}(X)},
\end{gather*}
here we use the fact $|F_{u^{\ast}(A)}|=|F_{A}|$ since $F_{u^{\ast}(A)}=u\circ F_{A}\circ u^{-1}$.

If $n\geq4$,
\begin{gather*}
\|a\wedge a\|_{L^{2}(X)}\leq C\|a\|^{2}_{L^{4}(X)} \leq C\|a\|^{2}_{L^{n}(X)}\leq C\|a\|^{2}_{L^{\frac{n}{2}}_{1}(X)}\\
\hphantom{\|a\wedge a\|_{L^{2}(X)}}{}
\leq C\|F_{A}\|^{2}_{L^{\frac{n}{2}}(X)} \leq C\|F_{A}\|^{2}_{L^{\infty}(X)} \leq C\|F_{A}\|^{2}_{L^{2}(X)},
\end{gather*}
here we apply the Sobolev embedding $L^{\frac{n}{2}}_{1}\hookrightarrow L^{n}$.

If $n=2,3$, \begin{gather*}
\|a\wedge a\|_{L^{2}(X)}\leq C\|a\|^{2}_{L^{4}(X)}\leq C\|a\|^{2}_{L^{2}_{1}(X)}\leq C\|F_{A}\|^{2}_{L^{2}(X)},
\end{gather*}
here we apply the Sobolev embedding $L^{2}_{1}\hookrightarrow L^{4}$.

Combining the preceding inequalities we have
\begin{gather*}\|F_{A}\|^{2}_{L^{2}(X)}\leq C\|F_{A}\|^{3}_{L^{2}(X)}.\end{gather*}
We can choose $\|F_{A}\|_{L^{2}(X)}$ suf\/f\/iciently small such that $C\|F_{A}\|_{L^{2}(X)}<1$, hence $\|F_{A}\|_{L^{2}(X)}\equiv0$ and thus $A$ must be a~f\/lat connection. Then we complete the proof.
\end{proof}

\section[Eigenvalue bounds for Laplacian $\Delta_{A}$]{Eigenvalue bounds for Laplacian $\boldsymbol{\Delta_{A}}$}\label{section3}
In this section, we will show that the least eigenvalue $\lambda(\Gamma)$ of $\dd^{\ast}_{\Gamma}\dd_{\Gamma}+\dd_{\Gamma}\dd^{\ast}_{\Gamma}$ has a positive lower bound $\lambda$ that is uniform with respect to $[\Gamma]\in M(P,g)$ under the given conditions on~$X$ and~$P$. The method is similar to Feehan~\cite{Feehan2014.12} that has proved the least eigenvalue $\mu_{g}(A)$ of~$\dd^{+,g}_{A}\dd_{A}^{+,\ast_{g}}$ which has a~positive lower bound~$\mu_{0}$ that is uniform with respect to $[A]\in\mathcal{B}(P,g)$ obeys $\|F^{+,g}_{A}\|_{L^{2}}\leq\varepsilon$, for a small enough $\varepsilon$ and under the given sets of conditions on~$g$,~$G$,~$P$ and~$X$.
\subsection[Continuity for the least eigenvalue of $\Delta_{A}$]{Continuity for the least eigenvalue of $\boldsymbol{\Delta_{A}}$}\label{section3.1}
From Uhlenbeck compactness theorem \cite{Uhlenbeck1982,Wehrheim}, we know
\begin{Proposition}
Let $G$ be a compact Lie group, $P$ be a $G$-bundle over a closed, smooth manifold~$X$ of dimension $n\geq2$ and endowed with a smooth Riemannian metric~$g$. Then the moduli space $M(P,g)$ is compact.
\end{Proposition}
The def\/inition of the least eigenvalue of $\Delta_{A}$ on $L^{2}(X,\Omega^{1}(\mathfrak{g}_{P}))$ is similar to \cite[Def\/inition~3.1]{Taubes1982}.
\begin{Definition}[least eigenvalue of $\Delta_{A}$]\label{C7}
Let $G$ be a compact Lie group, $P$ be a $G$-bundle over a~closed, smooth manifold~$X$ of dimension $n\geq2$ and endowed with a~smooth Riemannian metric~$g$. Let~$A$ be a connection of Sobolev class $L^{2}_{1}$ on $P$. The least eigenvalue of $\Delta_{A}$ on $L^{2}(X,\Omega^{1}(\mathfrak{g}_{P}))$ is
\begin{gather}
\lambda(A):=\inf_{v\in\Omega^{1}(\mathfrak{g}_{P})\backslash\{0\}}\frac{\langle\Delta_{A}v,v\rangle_{L^{2}}}{\|v\|^{2}}.
\end{gather}
\end{Definition}
The method to prove the continuity of the least eigenvalue of $\Delta_{A}$ with respect to the connection is similar to one by means of which Feehan proved continuity of the least eigenvalue of~$\dd^{+}_{A}\dd^{+,\ast}_{A}$ with respect to the connection in~\cite{Feehan2014.09,Feehan2014.12}.

We give a priori estimate for $v\in\Omega^{1}(X,\mathfrak{g}_{P})$ when the curvature $F_{A}$ is suf\/f\/iciently small in some $L^{p}$-norms.
\begin{Lemma}\label{C3}
Let $G$ be a compact Lie group, $P$ be a $G$-bundle over a closed, smooth mani\-fold~$X$ of dimension $n\geq2$ and endowed with a smooth Riemannian metric~$g$, let $2p\geq n$ when $n\geq3$ or $p>1$ when $n=2$. Then there are positive constants, $c=c(g,p)$ and $\varepsilon=\varepsilon(g,p)\in(0,1]$, with the following significance. If~$A$ is a connection on $P$ over $X$ such that
\begin{gather}\label{C2}
\|F_{A}\|_{L^{p}(X)}\leq\varepsilon,
\end{gather}
and $v\in\Omega^{1}(X,\mathfrak{g}_{P})$, then
\begin{gather}\label{C1}
\|v\|^{2}_{L^{2}_{1}(X)}\leq c\big(\|\dd_{A}v\|^{2}_{L^{2}(X)}+\|\dd^{\ast}_{A}v\|^{2}_{L^{2}(X)}+\|v\|^{2}_{L^{2}(X)}\big).
\end{gather}
\end{Lemma}
\begin{proof}The Weitzenb\"{o}ck formula for $v\in\Omega^{1}(X,\mathfrak{g}_{P})$, namely,
\begin{gather*}(\dd_{A}\dd^{\ast}_{A}+\dd^{\ast}_{A}\dd_{A})v=\nabla^{\ast}_{A}\nabla_{A}v+\operatorname{Ric}\circ v+\ast[\ast F_{A},v].\end{gather*}
Hence
\begin{gather*}\|\nabla_{A}v\|^{2}_{L^{2}(X)}\leq\|\dd^{\ast}_{A}v\|^{2}_{L^{2}(X)}+\|\dd_{A}v\|^{2}_{L^{2}(X)}+c\|v\|^{2}_{L^{2}(X)}
+|\langle\ast[\ast F_{A},v],v\rangle_{L^{2}(X)}|,\end{gather*}
where $c=c(g)$.

If $n\geq3$, by H\"{o}lder inequality, we see that
\begin{gather*}
|\langle\ast[\ast F_{A},v],v\rangle_{L^{2}(X)}|\leq\|F_{A}\|_{L^{n/2}(X)}\|v\|^{2}_{L^{2n/(n-2)}(X)}\\
\hphantom{|\langle\ast[\ast F_{A},v],v\rangle_{L^{2}(X)}|}{}
\leq c\|F_{A}\|_{L^{n/2}(X)}\|v\|^{2}_{L^{2}_{1}(X)}\leq c\|F_{A}\|_{L^{p}(X)}\|v\|^{2}_{L^{2}_{1}(X)},
\end{gather*}
for some $c=c(g)$.

If $n=2$, def\/ine $q\in(1,\infty)$ by $1/q=1-1/p$, we have{\samepage
\begin{gather*}
|\langle\ast[\ast F_{A},v],v\rangle_{L^{2}(X)}|\leq\|F_{A}\|_{L^{p}(X)}\|v\|^{2}_{L^{2q}(X)} \leq c\|F_{A}\|_{L^{p}(X)}\|v\|^{2}_{L^{2}_{1}(X)},
\end{gather*}
here we use the Sobolev embedding $L^{2}_{1}\hookrightarrow L^{2q}$.}

Combining of the preceding inequalities and Kato inequality $|\nabla|v||\leq|\nabla_{A}v|$ yields
\begin{gather*}
\|v\|^{2}_{L^{2}_{1}(X)}\leq\big(\|\nabla_{A}v\|^{2}_{L^{2}(X)}+\|v\|^{2}_{L^{2}(X)}\big)\\
\hphantom{\|v\|^{2}_{L^{2}_{1}(X)}}{} \leq\|\dd^{\ast}_{A}v\|^{2}_{L^{2}(X)}+\|\dd_{A}v\|^{2}_{L^{2}(X)}+(c+1)\|v\|_{L^{2}(X)}+c{\|F_{A}\|_{L^{p}(X)}}\|v\|^{2}_{L^{2}_{1}(X)},
\end{gather*}
for some $c=c(g)$. Provided ${c\|F_{A}\|_{L^{p}(X)}\leq1/2}$, rearrangements gives~(\ref{C1}).
\end{proof}

Following the idea of \cite[Lemma~35.12]{Feehan2014.09}, we also have a useful lemma.
\begin{Lemma}[$L^{2p}$-continuity of least eigenvalue of $\Delta_{A}$ with respect to the connection]\label{C4}
Let $G$ be a compact Lie group, $P$ be a $G$-bundle over a closed, smooth manifold $X$ of dimension $n\geq2$ and endowed with a smooth Riemannian metric~$g$, let $2p\geq n$ or $n\geq3$ and $p>1$ when $n=2$. Then there are positive constants, $C=C(g,p)$ and $\varepsilon=\varepsilon(g,p)$, with the following significance. If $A_{0}$, $A$ are smooth connections on~$P$ that obey the curvatures bounded~\eqref{C2} and
\begin{gather*}\|A-A_{0}\|_{L^{2p}(X)}\leq\varepsilon,\end{gather*}
then, we denote $a:=A-A_{0}$,
\begin{gather*}
\big(1\!-\!C\|a\|^{2}_{L^{2p}(X)}\big)\lambda(A_{0})-C\|a\|^{2}_{L^{2p}(X)}\leq\lambda(A)\leq\big(1-C\|a\|^{2}_{L^{2p}(X)}\big)^{-1}\big(\lambda(A_{0})
\!+\!C\|a\|^{2}_{L^{2p}(X)}\big).
\end{gather*}
\end{Lemma}
\begin{proof}
For convenience, write $a:=A-A_{0}\in L^{n}(X,\Omega^{1}\otimes\mathfrak{g}_{P})$. For $v\in L^{2}_{1}(X,\Omega^{1}\otimes\mathfrak{g}_{P})$, we have $\dd_{A}v=\dd_{A_{0}}v+[a,v]$.

If $n\geq3$, by H\"{o}lder inequality
\begin{gather*}
\|\dd_{A}v\|^{2}_{L^{2}(X)}=\|\dd_{A_{0}}v+[a,v]\|^{2}_{L^{2}(X)} \geq\|\dd_{A_{0}}v\|^{2}_{L^{2}(X)}-\|[a,v]\|^{2}_{L^{2}(X)}\\
\hphantom{\|\dd_{A}v\|^{2}_{L^{2}(X)}}{}
\geq\|\dd_{A_{0}}v\|^{2}_{L^{2}(X)}-2{\|a\|^{2}_{L^{n}(X)}}\|v\|^{2}_{L^{2n/(n-2)}(X)} \\
\hphantom{\|\dd_{A}v\|^{2}_{L^{2}(X)}}{}
\geq\|\dd_{A_{0}}v\|^{2}_{L^{2}(X)}-2c_{1}{\|a\|^{2}_{L^{n}(X)}}\|v\|^{2}_{L^{2}_{1,A_{0}}(X)}\\
\hphantom{\|\dd_{A}v\|^{2}_{L^{2}(X)}}{}
\geq\|\dd_{A_{0}}v\|^{2}_{L^{2}(X)}-2c_{1}{\|a\|^{2}_{L^{2p}(X)}}\|v\|^{2}_{L^{2}_{1,A_{0}}(X)},
\end{gather*}
where $c_{1}=c_{1}(g)$ is the Sobolev embedding constant for $L^{2}_{1}\hookrightarrow L^{2n/(n-2)}$.

Similarly, $\dd^{\ast}_{A}v=\dd^{\ast}_{A_{0}}v\pm\ast[a,\ast v]$ and
\begin{gather*}
\|\dd^{\ast}_{A}v\|^{2}_{L^{2}(X)}=\|\dd^{\ast}_{A_{0}}v\pm\ast[a,\ast v]\|^{2}_{L^{2}(X)} \geq\|\dd^{\ast}_{A_{0}}v\|^{2}_{L^{2}(X)}-\|[a,\ast v]\|^{2}_{L^{2}(X)}\\
\hphantom{\|\dd^{\ast}_{A}v\|^{2}_{L^{2}(X)}}{}
\geq\|\dd^{\ast}_{A_{0}}v\|^{2}_{L^{2}(X)}-2\|a\|^{2}_{L^{n}(X)}\|v\|^{2}_{L^{2n/(n-2)}(X)} \\
\hphantom{\|\dd^{\ast}_{A}v\|^{2}_{L^{2}(X)}}{}
\geq\|\dd^{\ast}_{A_{0}}v\|^{2}_{L^{2}(X)}-2c_{1}\|a\|^{2}_{L^{n}(X)}\|v\|^{2}_{L^{2}_{1,A_{0}}(X)}\\
\hphantom{\|\dd^{\ast}_{A}v\|^{2}_{L^{2}(X)}}{}
\geq\|\dd_{A_{0}}v\|^{2}_{L^{2}(X)}-2c_{1}{\|a\|^{2}_{L^{2p}(X)}}\|v\|^{2}_{L^{2}_{1,A_{0}}(X)}.
\end{gather*}
If $n=2$, def\/ine $q\in(1,\infty)$ by $1=1/p+1/q$,
\begin{gather*}
\|\dd_{A}v\|^{2}_{L^{2}(X)}=\|\dd_{A_{0}}v+[a,v]\|^{2}_{L^{2}(X)} \geq\|\dd_{A_{0}}v\|^{2}_{L^{2}(X)}-\|[a,v]\|^{2}_{L^{2}(X)}\\
\hphantom{\|\dd_{A}v\|^{2}_{L^{2}(X)}}{}
\geq\|\dd_{A_{0}}v\|^{2}_{L^{2}(X)}-2{\|a\|^{2}_{L^{2p}(X)}}\|v\|^{2}_{L^{2q}(X)} \\
\hphantom{\|\dd_{A}v\|^{2}_{L^{2}(X)}}{}
\geq\|\dd_{A_{0}}v\|^{2}_{L^{2}(X)}-2c_{1}{\|a\|^{2}_{L^{2p}(X)}}\|v\|^{2}_{L^{2}_{1,A_{0}}(X)},
\end{gather*}
here we use the Sobolev embedding $L^{2}_{1}\hookrightarrow L^{2q}$.

Applying the a priori estimate (\ref{C1}) for $\|v\|_{L^{2}_{1}(X)}$ from Lemma~\ref{C3}, with $c=c(g)$ and smooth enough $\varepsilon=\varepsilon(g)$, we get
\begin{gather*}\|v\|^{2}_{L^{2}_{1}(X)}\leq c\big(\|\dd_{A_{0}}v\|^{2}_{L^{2}(X)}+\|\dd^{\ast}_{A_{0}}v\|^{2}_{L^{2}(X)}+\|v\|^{2}_{L^{2}(X)}\big).\end{gather*}
Combining of the preceding inequalities gives
\begin{gather*}
\|\dd_{A}v\|^{2}_{L^{2}(X)}+\|\dd^{\ast}_{A}v\|^{2}_{L^{2}(X)}\geq\big(\|\dd_{A_{0}}v\|^{2}_{L^{2}(X)}
+\|\dd^{\ast}_{A_{0}}v\|^{2}_{L^{2}(X)}\big)-4cc_{1}\|a\|^{2}_{L^{2p}(X)}\|v\|^{2}_{L^{2}(X)}\\
\hphantom{\|\dd_{A}v\|^{2}_{L^{2}(X)}+\|\dd^{\ast}_{A}v\|^{2}_{L^{2}(X)}\geq}{}
-4c_{1}c\|a\|^{2}_{L^{2p}(X)}\big(\|\dd_{A_{0}}v\|^{2}_{L^{2}(X)}+\|\dd^{\ast}_{A_{0}}v\|^{2}_{L^{2}(X)}\big).
\end{gather*}
Now take $v$ to be an eigenvalue of $\Delta_{A}$ with eigenvalue $\lambda(A)$ and $\|v\|_{L^{2}(X)}=1$ and also suppose that {$\|A-A_{0}\|_{L^{2p}(X)}$} is small enough that {$4c_{1}c\|a\|^{2}_{L^{2p}(X)}\leq1/2$}. The preceding inequality then gives
\begin{gather*}
\lambda(A)\geq\big(1-4c_{1}c{\|a\|^{2}_{L^{2p}(X)}}\big)\big(\|\dd_{A_{0}}v\|^{2}_{L^{2}(X)}
+\|\dd^{\ast}_{A_{0}}v\|^{2}_{L^{2}(X)}\big)-4c_{1}c\|a\|^{2}_{L^{2p}(X)}.
\end{gather*}

Since $\|v\|_{L^{2}(X)}=1$, we have $(\|\dd_{A_{0}}v\|^{2}_{L^{2}(X)} +\|\dd^{\ast}_{A_{0}}v\|^{2}_{L^{2}(X)})\geq\lambda(A_{0})$, hence
\begin{gather*}
\lambda(A)\geq\big(1-{4c_{1}c}{\|a\|^{2}_{L^{2p}(X)}}\big)\lambda(A_{0})-4c_{1}c\|a\|^{2}_{L^{2p}(X)}.
\end{gather*}
To obtain the upper bounded for $\lambda(A)$, we exchange the roles of $A$ and $A_{0}$ that yields the inequality,
\begin{gather*}
\lambda(A_{0})\geq\big(1-{4c_{1}c}{\|a\|^{2}_{L^{2p}(X)}}\big)\lambda(A)-4c_{1}c\|a\|^{2}_{L^{2p}(X)}.\tag*{\qed}
\end{gather*}\renewcommand{\qed}{}
\end{proof}

\subsection[Uniform positive lower bound for the least eigenvalue of $\Delta_{A}$]{Uniform positive lower bound for the least eigenvalue of $\boldsymbol{\Delta_{A}}$}\label{section3.2}
Our results in Section~\ref{section3.1} assure the continuity of $\lambda(\cdot)$ with respect to the Uhlenbeck topology, and they will be applied here. Before doing this, we recall
\begin{Definition}[{\cite[Def\/inition 2.4]{Donaldson}}]\label{D3.6}
Let $G$ be a compact Lie group, $P$ be a $G$-bundle over a~closed, smooth manifold $X$ of dimension $n\geq2$ and endowed with a smooth Riemannian metric~$g$. The f\/lat connection, $\Gamma$, is \emph{non-degenerate} if
\begin{gather*}
\ker{\Delta_{\Gamma}}|_{\Omega^{1}(X,\mathfrak{g}_{P})}=\{0\}.
\end{gather*}
\end{Definition}
Then we use the results of the continuous of $\lambda[\cdot]$ and compactness of $M(P,g)$ to prove that~$\lambda[\cdot]$ has a uniform lower positive bound.
\begin{Proposition}\label{P4}
Let $G$ be a compact Lie group, $P$ be a $G$-bundle over a closed, smooth mani\-fold~$X$ of dimension $n\geq2$ and endowed with a smooth Riemannian metric~$g$. Then there is a positive constant~$\lambda$ with the following significance. Suppose all flat connections on~$P$ are non-degenerate. If $\Gamma$ is a flat connection, then
\begin{gather*}\lambda(\Gamma)\geq\lambda,\end{gather*}
where $\lambda(\Gamma)$ is as in Definition~{\rm \ref{C7}}.
\end{Proposition}
\begin{proof} The conclusion is a consequence of the fact that $M(P,g)$ is compact, \begin{gather*}\lambda[\cdot]\colon \ M(P,g)\ni[\Gamma]\rightarrow \lambda(\Gamma)\in[0,\infty),\end{gather*}
to $M(P,g)$ is continuous by Lemma~\ref{C4}, the fact that $\lambda(\Gamma)>0$ for $[\Gamma]\in M(P,g).$
\end{proof}

We consider the open subset of the space $\mathcal{B}(P,g)$ def\/ined by
\begin{gather*}{\mathcal{B}_{\varepsilon}(P,g):=\{[A]\in\mathcal{B}(P,g)\colon \|F_{A}\|_{L^{p}(X)}<\varepsilon\},}\end{gather*}
where $p$ is a constant such that $2p>n$. Then we have
\begin{Theorem}\label{T3.10}
Let $G$ be a compact Lie group, $P$ be a $G$-bundle over a closed, smooth mani\-fold~$X$ of dimension $n\geq2$ and endowed with a smooth Riemannian metric~$g$, and $2p>n$. Then there is a positive constant $\varepsilon=\varepsilon(g,n)$ with the following significance. Suppose all flat connections on~$P$ are non-degenerate. If~$A$ is a smooth connection on $P$ such that
\begin{gather*}\|F_{A}\|_{L^{p}(X)}\leq\varepsilon,\end{gather*}
and $\lambda(A)$ is as in Definition~{\rm \ref{C7}}, then
\begin{gather*}\lambda(A)\geq\lambda/2,\end{gather*}
where $\lambda$ is the constant in Proposition~{\rm \ref{P4}}.
\end{Theorem}
\begin{proof}
For a smooth connection $A$ on $P$ with $\|F_{A}\|_{L^{p}(X)}\leq\varepsilon$, where $\varepsilon$ is as in the hypotheses of Theorem~\ref{T2}. Then there exists a flat connection $\Gamma$ on $P$ and a gauge transformation $g\in L^{p}_{2}(X)$ such that
\begin{gather*}\|g^{\ast}(A)-\Gamma\|_{L^{p}_{1,\Gamma}(X)}\leq C\|F_{A}\|_{L^{p}(X)}.\end{gather*}
For $\|F_{A}\|_{L^{p}(X)}$ suf\/f\/iciently small, we can apply Lemma~\ref{C4} for $A$ and $\Gamma$ to obtain
\begin{gather*}
\lambda(A)\geq\big(1-c\|g^{\ast}(A)-\Gamma\|_{L^{2p}(X)}\big)\lambda(\Gamma)-c\|g^{\ast}(A)-\Gamma\|_{L^{2p}(X)}\\
\hphantom{\lambda(A)}{}
\geq\big(1-c\|g^{\ast}(A)-\Gamma\|_{L^{p}_{1,\Gamma}(X)}\big)\lambda(\Gamma)-c\|g^{\ast}(A)-\Gamma\|_{L^{p}_{1,\Gamma}(X)}\\
\hphantom{\lambda(A)}{}
\geq\lambda-Cc\|F_{A}\|_{L^{p}(X)}(1+\lambda),
\end{gather*}
here we use Sobolev embedding $L^{p}_{1}\hookrightarrow L^{2p}$. We choose $\|F_{A}\|_{L^{p}(X)}$ suf\/f\/iciently small such that $|F_{A}\|_{L^{p}(X)}\leq\frac{\lambda}{2Cc(1+\lambda)}$, then we have $\lambda(A)\geq\lambda/2$.
\end{proof}

\subsection{The case of dimension four}\label{section3.3}
In this section, we will show a theorem similar to Theorem \ref{T3.10} in the case of dimension four, but we only need to suppose that $F_{A}$ with $L^{2}$-norm is suf\/f\/iciently small. F\/irst, we recall a priori the $L^{p}$ estimate for the connection Laplace operator which was proved by Feehan.
\begin{Lemma}[{\cite[Lemma 35.5]{Feehan2014.09}}]\label{L3.8} Let $X$ be a smooth manifold $X$ of dimension $n\geq4$ and endowed with a smooth Riemannian metric~$g$ and $q\in(n,\infty)$. Then there is a positive constant $c=c(g,p)$ with the following significance. Let $r\in(\frac{n}{3},\frac{n}{2})$ be defined by $1/r=2/n+1/q$. Let $A$ is a $C^{\infty}$ connection on a vector bundle~$E$ over~$X$. If $v\in C^{\infty}(X,E)$, then
\begin{gather}\label{E3.4}
\|v\|_{L^{q}(X)}\leq c\big(\|\nabla^{\ast}_{A}\nabla_{A}v\|_{L^{r}(X)}+\|v\|_{L^{r}(X)}\big).
\end{gather}
\end{Lemma}

We now apply Lemma \ref{L3.8} to $E=\Omega^{1}\otimes\mathfrak{g}_{P}$, then we give a priori $L^{p}$-estimate for~$\Delta_{A}$.
\begin{Lemma}Let $G$ be a compact Lie group, $P$ be a $G$-bundle over a closed, smooth mani\-fold~$X$ of dimension $n\geq4$ and endowed with a smooth Riemannian metric~$g$ and $q\in(n,\infty)$. Then there are positive constants, $c=c(g)$ and $\varepsilon=\varepsilon(g)$, with the following significance. Let $r\in(\frac{n}{3},\frac{n}{2})$ defined by $1/r=2/n+1/q$. Let $A$ be a~$C^{\infty}$ connection on $P$ that each obeys the curvature bounded~\eqref{C2}. If $v\in\Omega^{1}(X,\mathfrak{g}_{P})$, then
\begin{gather}\label{E3.5}
\|v\|_{L^{q}(X)}\leq c\big(\|\Delta_{A}v\|_{L^{r}(X)}+\|v\|_{L^{r}(X)}\big).
\end{gather}
\end{Lemma}
\begin{proof}
For $v\in\Omega^{1}(X,\mathfrak{g}_{P})$, from the Weitzenb\"{o}ck formula, we have
\begin{gather*}\Delta_{A}v=\nabla^{\ast}_{A}\nabla_{A}v+\operatorname{Ric}\cdot v+\{v,F_{A}\}.\end{gather*}
Hence
\begin{gather*}\|\Delta_{A}v\|_{L^{r}(X)}\leq\|\nabla_{A}^{\ast}\nabla_{A}v\|_{L^{r}(X)}+c\|v\|_{L^{r}(X)}+\|\{v,F_{A}\}\|_{L^{r}(X)},\end{gather*}
for some $c=c(g)$. Since $1/r=2/n+1/q$ by hypothesis, we see that
\begin{gather*}\|\{v,F_{A}\}\|_{L^{r}(X)}\leq c\|F_{A}\|_{L^{n/2}(X)}\|v\|_{L^{q}(X)},\end{gather*}
for some $c=c(g)$. Combining the preceding inequalities with the equation (\ref{E3.4}) we get
\begin{gather*}\|v\|_{L^{q}(X)}\leq\|\nabla_{A}^{\ast}\nabla_{A}v\|_{L^{r}(X)}+c\|v\|_{L^{r}(X)}+c\|F_{A}\|_{L^{n/2}(X)}\|v\|_{L^{q}(X)},\end{gather*}
for some $c=c(g,q)$. Provided $c\|F_{A}\|_{L^{n/2}(X)}\leq1/2$, the rearrangement gives (\ref{E3.4}).
\end{proof}

Hence, following the idea of \cite[Lemma~35.13]{Feehan2015}, we also have a useful lemma.
\begin{Lemma}[$L^{p}$ continuity of least eigenvalue of $\Delta_{A}$ with respect to the connection for $2<p\leq4$]\label{L3.13}
Let $G$ be a compact Lie group, $P$ be a $G$-bundle over a closed, smooth four-mani\-fold~$X$ and endowed with a smooth Riemannian metric~$g$. Then there are positive constants, $c=c(g)$ and $\varepsilon=\varepsilon(g)$, with the following significance. If $A_{0}$ and $A$ are $C^{\infty}$ connections on $P$ that each obeys the curvature bounded~\eqref{C2}, we denote $a:=A-A_{0}$, then $\lambda(A)$ satisfies
\begin{gather*}\lambda(A)\geq\lambda(A_{0})-c_{0}\big(1+\lambda^{2}(A)\big)\|a\|^{2}_{L^{p}(X)},\end{gather*}
and
\begin{gather*}\lambda(A)\leq\lambda(A_{0})+c_{0}\big(1+\lambda^{2}(A_{0})\big)\|a\|^{2}_{L^{p}(X)}.\end{gather*}
\end{Lemma}
\begin{proof}
For convenience, write $a:=A-A_{0}\in L^{p}(X,\Omega^{1}\otimes\mathfrak{g}_{P})$. Def\/ine $q\in[2,\infty)$ by $1/2=1/p+1/q$ and consider $v\in L^{2}_{1}(X,\Omega^{1}\otimes\mathfrak{g}_{P})$. We use $\dd_{A}v=\dd_{A_{0}}v+[a,v]$ and the H\"{o}ler inequalities to give
\begin{gather*}
\|\dd_{A}v\|^{2}_{L^{2}(X)}=\|\dd_{A_{0}}v+[a,v]\|^{2}_{L^{2}(X)} \geq\|\dd_{A_{0}}v\|^{2}_{L^{2}(X)}-\|[a,v]\|^{2}_{L^{2}(X)}\\
\hphantom{\|\dd_{A}v\|^{2}_{L^{2}(X)}}{}
\geq\|\dd_{A_{0}}v\|^{2}_{L^{2}(X)}-2\|a\|^{2}_{L^{p}(X)}\|v\|^{2}_{L^{q}(X)}.
\end{gather*}
Similarly, $\dd^{\ast}_{A}v=\dd^{\ast}_{A_{0}}v\pm\ast[a,\ast v]$ and
\begin{gather*}
\|\dd^{\ast}_{A}v\|^{2}_{L^{2}(X)}=\|\dd^{\ast}_{A_{0}}v\pm\ast[a,\ast v]\|^{2}_{L^{2}(X)} \geq\|\dd^{\ast}_{A_{0}}v\|^{2}_{L^{2}(X)}-\|\ast[a,\ast v]\|^{2}_{L^{2}(X)}\\
\hphantom{\|\dd^{\ast}_{A}v\|^{2}_{L^{2}(X)}}{}
\geq\|\dd^{\ast}_{A_{0}}v\|^{2}_{L^{2}(X)}-2\|a\|^{2}_{L^{p}(X)}\|v\|^{2}_{L^{q}(X)}.
\end{gather*}
For $p>4$, we have $2\leq q<4$ and $\|v\|_{L^{q}(X)}\leq({\rm vol}(X))^{1/q-1/4}\|v\|_{L^{4}(X)}$, while for $2<p\leq4$, we have $4\leq q<\infty$. Therefore, it suf\/f\/ices to consider the case $4\leq q<\infty$. Applying the a priori estimate~(\ref{E3.4}) and $r\in(4/3,2)$ def\/ined by $1/r=1/2+1/q$, we get
\begin{gather*}
\|v\|^{2}_{L^{q}(X)}\leq c\big(\|\Delta_{A}v\|^{2}_{L^{4/3}(X)}+\|v\|^{2}_{L^{r}(X)}\big)
\leq c\big(\|\Delta_{A}v\|^{2}_{L^{2}(X)}+\|v\|^{2}_{L^{r}(X)}\big).
\end{gather*}
Combining the preceding inequalities, we get
\begin{gather*}
\begin{split}
& \|\dd_{A}v\|^{2}_{L^{2}(X)}+\|\dd^{\ast}_{A}v\|^{2}_{L^{2}(X)}\\
& \qquad{} \geq\|\dd_{A_{0}}v\|^{2}_{L^{2}(X)}+\|\dd^{\ast}_{A_{0}}v\|^{2}_{L^{2}(X)}
-2c_{1}\|a\|^{2}_{L^{p}(X)}\big(\|\Delta_{A}v\|^{2}_{L^{r}(X)}+\|v\|^{2}_{L^{r}(X)}\big)\\
& \qquad{} \geq\|\dd_{A_{0}}v\|^{2}_{L^{2}(X)}+\|\dd^{\ast}_{A_{0}}v\|^{2}_{L^{2}(X)}
-2c_{0}\|a\|^{2}_{L^{p}(X)}\big(\|\Delta_{A}v\|_{L^{2}(X)}+\|v\|^{2}_{L^{2}(X)}\big),
\end{split}
\end{gather*}
for $c_{0}=c_{0}(p,q)=2c_{1}{\rm vol}(X)^{2/q}$, using the fact that $\|v\|_{L^{r}(X)}\leq {\rm vol}(X)^{1/q}\|v\|_{L^{2}(X)}$ for $r\in(4/3,2)$ and $1/r=1/2+1/q$. By taking $v\in L^{2}_{1}(X,\Omega^{1}\otimes\mathfrak{g}_{P})$ to be an eigenvector of $\Delta_{A}$ with eigenvalue $\lambda(A)$ such that $\|v\|_{L^{2}(X)}=1$ and noting that $\|\Delta_{A}v\|_{L^{2}(X)}=\lambda(A)$ we obtain the bound
\begin{gather*}
\lambda(A)\geq\|\dd_{A_{0}}v\|^{2}_{L^{2}(X)}+\|\dd_{A_{0}}^{\ast}v\|^{2}_{L^{2}(X)}-c_{0}\big(1+\lambda^{2}(A)\big)\|a\|^{2}_{L^{p}(X)}.\end{gather*}
But $\|\dd_{A_{0}}v\|^{2}_{L^{2}(X)}+\|\dd_{A_{0}}^{\ast}v\|^{2}_{L^{2}(X)}\geq\lambda(A_{0})$ and thus we have the inequality
\begin{gather*}\lambda(A)\geq\lambda(A_{0})-c_{0}\big(1+\lambda^{2}(A)\big)\|a\|^{2}_{L^{p}(X)}.\end{gather*}
Interchanging of the roles $A$ and $A_{0}$ in the preceding derivation yields
\begin{gather*}\lambda(A)\leq\lambda(A_{0})+c_{0}\big(1+\lambda^{2}(A_{0})\big)\|a\|^{2}_{L^{p}(X)}.\tag*{\qed}\end{gather*}\renewcommand{\qed}{}
\end{proof}

We consider a sequence of $C^{\infty}$ connections $\{A_{i}\}_{i\in\mathbb{N}}$ on $P$ such that $\sup\|F_{A_{i}}\|_{L^{2}(X)}<\infty$. We denote
\begin{gather*}\Sigma=\big\{x\in X\colon \lim_{r\searrow0}\limsup_{i\rightarrow\infty}\|F_{A_{i}}\|^{2}_{L^{2}(B_{r})(x)}\geq\tilde{\varepsilon}\big\},\end{gather*}
the constant $\tilde{\varepsilon}\in(0,1]$ as in \cite[Theorem~3.2]{Sedlacek}. We can see $\Sigma$ is a f\/inite points $\{x_{1},\ldots,x_{L}\}$ in~$X$. In our article, we consider the open subset of the space $\mathcal{B}(P,g)$ def\/ined by
\begin{gather*}\mathcal{B}_{\varepsilon}=\{[A]\in\mathfrak{B}(P,g)\colon \|F_{A}\|_{L^{2}(X)}<\varepsilon\}.\end{gather*}
We can choose $\varepsilon$ suf\/f\/iciently small such that any sequence $\{A_{i}\}_{i\in\mathbb{N}}$ has the empty set~$\Sigma$. From \cite[Theorem~3.1]{Sedlacek} and \cite[Theorem~35.15]{Feehan2014.09}, we have

\begin{Theorem}Let $G$ be a compact Lie group and $P$ be a principal $G$-bundle over a close, smooth four-dimensional $X$ with Riemannian metric~$g$. If $\{A_{i}\}_{i\in{\mathbb N}}$ is a sequence of $C^{\infty}$ connections on~$P$ such that $\|F_{A_{i}}\|_{L^{2}(X)}\leq\varepsilon$, there exists a subsequence, a countable set of arbitrarily small geodesic balls $\{B_{\alpha}\}_{\alpha\in\mathbb{N}}$ covering $X$, $C^{\infty}$-sections
\begin{gather*}\sigma_{\alpha,i}\colon \ B_{\alpha}\rightarrow P, \qquad A_{\alpha}\in L^{2}_{1}\big(B_{\alpha};\Omega^{1}B_{\alpha}\otimes\mathfrak{g}\big), \qquad g_{\alpha\beta}\in L^{4}_{1}(B_{\alpha}\cap B_{\beta};G),\end{gather*}
such that
\begin{enumerate}\itemsep=0pt
\item[$(1)$] $\dd^{\ast}A_{\alpha}(i)=0$, for all $i$ sufficiently large,
\item[$(2)$] $\dd^{\ast}A_{\alpha}=0$,
\item[$(3)$] $g_{\alpha\beta}(i)\rightharpoonup g_{\alpha\beta}$ weakly in $L^{4}_{1}(B_{\alpha}\cap B_{\beta};G)$,
\item[$(4)$] $F_{\alpha}(i)\rightharpoonup F_{\alpha}$ weakly in ${L^{2}({B_{\alpha};\Omega^{2}B_{\alpha}\otimes\mathfrak{g}})}$,
\item[$(5)$] the sequence $\{A_{\alpha}(i)\}_{i\in\mathbb{N}}$ obeys
\begin{enumerate}\itemsep=0pt
\item[$(a)$] $A_{\alpha}(i)\rightharpoonup A_{\alpha}$ weakly in ${L^{2}_{1}(B_{\alpha};\Omega^{1}B_{\alpha}\otimes\mathfrak{g})}$, and
\item[$(b)$] $A_{\alpha}(i)\rightarrow A_{\alpha}$ strongly in ${L^{p}(B_{\alpha};\Omega^{1}B_{\alpha}\otimes\mathfrak{g})}$ for $1\leq p<4$,
\end{enumerate}
\item[$(6)$] $A_{\alpha}=g_{\alpha\beta}^{-1}A_{\beta}g_{\alpha\beta}+g_{\alpha\beta}^{-1}dg_{\alpha\beta}$.
\end{enumerate}
Here $A_{\alpha}(i)=\sigma^{\ast}_{\alpha}A_{i}$, $F_{\alpha}=\dd A_{\alpha}+[A_{\alpha},A_{\alpha}]$, $F_{\alpha}(i)=\dd A_{\alpha}(i)+[A_{\alpha}(i),A_{\alpha}(i)]$ and $\dd^{\ast}$ is the formal adjoint of $\dd$ with respect to the flat metric defined by a choice of geodesic normal coordinates on~$B_{\alpha}$.
\end{Theorem}

From \cite[Theorem 4.3]{Sedlacek} and \cite[Theorem 35.17]{Feehan2014.09}, we have
\begin{Theorem}\label{CY1}
Let $G$ be a compact Lie group and $P$ be a principal $G$-bundle over a close, smooth four-dimensional $X$ with Riemannian metric~$g$. If $\{A_{i}\}_{i\in{\mathbb N}}$ is a sequence of $C^{\infty}$ connections on $P$, in the sense that
\begin{gather*}{\rm YM}(A_{i})\searrow0 \qquad \text{as}\quad i\rightarrow\infty,\end{gather*}
then the following hold, for each $\alpha,\beta\in{\mathbb N}$,
\begin{enumerate}\itemsep=0pt
\item[$(1)$] $A_{\alpha}\in {C^{\infty}(B_{\alpha};\Omega^{1}B_{\alpha}\otimes\mathfrak{g})}$ and a solution to the flat connection,
\item[$(2)$] $g_{\alpha\beta}\in C^{\infty}(B_{\alpha}\bigcap B_{\beta};G)$,
\item[$(3)$] the sequence, $\{A_{\alpha}\}_{\alpha\in{\mathbb N}}$ and $\{g_{\alpha\beta}\}_{\alpha,\beta\in{\mathbb N}}$ define a $C^{\infty}$ flat connection $A_{\infty}$ on a principal $G$-bundle $P_{\infty}$ over $X$.
\end{enumerate}
\end{Theorem}
Then, we have the useful
\begin{Corollary}
Assume the hypotheses of Theorem~{\rm \ref{CY1}}, then
\begin{gather*}\lim_{i\rightarrow\infty}\lambda(A_{i})=\lambda(A_{\infty}),\end{gather*}
where $\lambda(\Gamma)$ is as in Definition~{\rm \ref{C7}}.
\end{Corollary}
\begin{proof}
From Theorem~\ref{CY1}, $\sigma^{\ast}_{\alpha,i}A_{i}\rightharpoonup\sigma^{\ast} A_{\infty}$ weakly in ${L^{2}_{1}(B_{\alpha},\Omega^{1}B_{\alpha}\otimes\mathfrak{g}_{P})}$. For $L^{2}_{1}\Subset L^{p}$, ($2<p<4$), hence
\begin{gather*}\|\sigma^{\ast}_{\alpha,i}A_{i}-\sigma^{\ast}A_{\infty}\|_{L^{p}(B_{\alpha})}\rightarrow0\qquad\text{as}\quad i\rightarrow\infty,\end{gather*}
for sequences of local sections $\{\sigma_{\alpha,m}\}_{m\in\mathbb{N}}$ of $P\upharpoonright B_{\alpha}$ and a local section $\sigma_{\alpha}$ of $P_{\infty}\upharpoonright B_{\alpha}$ and
\begin{gather*}\|A_{i}-A_{\infty}\|_{L^{p}(X)}\leq\sum_{a}\|\sigma^{\ast}_{\alpha,m}A_{m}-\sigma^{\ast}A_{\infty}\|_{L^{p}(B_{\alpha})}.\end{gather*}
Hence from Lemma \ref{L3.13}, we have
\begin{gather*}\lim_{i\rightarrow\infty}\lambda(A_{i})\geq\lambda(A_{\infty})-c_{0}\big(1+\lim_{i\rightarrow\infty}\lambda^{2}(A_{i})\big)\lim_{i\rightarrow\infty}\|a_{i}\|_{L^{p}(X)},\end{gather*}
and
\begin{gather*}\lim_{i\rightarrow\infty}\lambda(A_{i})\leq\lambda(A_{\infty})
+c_{0}\big(1+\lambda^{2}(A_{\infty})\big)\lim_{i\rightarrow\infty}\|a_{i}\|_{L^{p}(X)}.\end{gather*}
Then we obtain
\begin{gather*}\lim_{i\rightarrow\infty}\lambda(A_{i})=\lambda(A_{\infty}).\tag*{\qed}\end{gather*}\renewcommand{\qed}{}
\end{proof}

 Then we have
\begin{Theorem}\label{T3.16}Let $G$ be a compact Lie group, $P$ be a $G$-bundle over a closed, smooth $4$-mani\-fold and endowed with a smooth Riemannian metric~$g$. Then there is a positive constant $\varepsilon=\varepsilon(g)\in(0,1]$ with the following significance. Suppose all flat connections on~$P$ are non-degenerate. If~$A$ is a~$C^{\infty}$ connection on $P$ such that
\begin{gather*}
\|F_{A}\|_{L^{2}(X)}\leq\varepsilon,
\end{gather*}
and $\lambda(A)$ is as in Definition~{\rm \ref{C7}}, then
\begin{gather*}\lambda(A)\geq\frac{\lambda}{2},\end{gather*}
where $\lambda$ is the positive constant in Proposition~{\rm \ref{P4}}.
\end{Theorem}
\begin{proof}
Suppose that the constant $\lambda\in(0,1]$ does not exist. We may then choose a sequence $\{A_{i}\}_{i\in{\mathbb N}}$ of connections on~$P$ such that $\|F_{A_{i}}\|_{L^{2}(X)}\rightarrow0$ and $\lambda(A_{i})\rightarrow0$ as $i\rightarrow\infty$. Since $\lim\limits_{i\rightarrow\infty}\lambda(A_{i})=\lambda(A_{\infty})$ and $\lambda(A_{\infty})>0$ by $A_{\infty}$ is a f\/lat connection, then it contradicts our initial assumption regarding the sequence $\{A_{i}\}_{i\in{\mathbb N}}$.
\end{proof}

\section{Proof of the main Theorem \ref{T1.1}}\label{section4}
Now, we begin to prove the energy gap result for the complex Yang--Mills equations. At f\/irst, we prove the complex Yang--Mills equations will be reduce to pure Yang--Mills equation under the certain conditions for $g$, $G$, $P$, and~$X$.
\begin{Proposition}\label{P4.1}Let $X$ be a closed, oriented, smooth Riemannian manifold of dimension $n\geq2$ with smooth Riemannian metric~$g$, $P$ be a~$G$-bundle over $X$, let $2p>n$ when $n\neq 4$ or $p\geq2$ when $n=4$. Then there exists a positive constant ${\varepsilon=\varepsilon(g,n,p)}$ with the following significance. Suppose that all flat connections on $P$ are non-degenerate. If the pair $(A,\phi)$ is a~$C^{\infty}$-solution of complex Yang--Mills equations over $X$ and the curvature $F_{A}$ of connec\-tion~$A$ obeys
\begin{gather*}
\|F_{A}\|_{L^{p}(X)}\leq\varepsilon,
\end{gather*}
then $\phi$ vanishes.
\end{Proposition}
\begin{proof} If we suppose that
\begin{gather*}\|F_{A}\|_{L^{p}(X)}\leq\varepsilon,\end{gather*}
where $p$ and $\varepsilon$ are the constants satisfying the hypotheses in Theorems~\ref{T3.10} and~\ref{T3.16}, then there exists a positive constant $\lambda$ such that
\begin{gather*}\|\dd_{A}v\|^{2}_{L^{2}(X)}+\|\dd^{\ast}_{A}v\|^{2}_{L^{2}(X)}\geq\lambda/2\|v\|^{2}_{L^{2}(X)}, \qquad \forall\, v\in\Omega^{1}(X,\mathfrak{g}_{P}).\end{gather*}
We have an identity for the solution of complex Yang--Mills equations
\begin{gather*}
\|\dd_{A}\phi\|^{2}_{L^{2}(X)}+\|\dd^{\ast}_{A}\phi\|^{2}_{L^{2}(X)}+2\|\phi\wedge\phi\|^{2}_{L^{2}(X)}-2\langle F_{A},\phi\wedge\phi\rangle_{L^{2}(X)}=0.
\end{gather*}
Hence, we have
\begin{gather*}
\lambda/2\|\phi\|^{2}_{L^{2}(X)} \leq\|\dd_{A}\phi\|^{2}_{L^{2}(X)}+\|\dd^{\ast}_{A}\phi\|^{2}_{L^{2}(X)} \leq|\langle F_{A},\phi\wedge\phi\rangle_{L^{2}(X)}|\\
\hphantom{\lambda/2\|\phi\|^{2}_{L^{2}(X)}}{} \leq C\|F_{A}\|_{L^{p}(X)}\|\phi\|^{2}_{L^{2q}(X)}\leq C\|F_{A}\|_{L^{p}(X)}\|\phi\|^{2}_{L^{2}(X)},
\end{gather*}
where $1/q=1-1/p$, $C=C(g,p)$. We can choose $\|F_{A}\|_{L^{p}(X)}\leq\varepsilon$ suf\/f\/iciently small such that $C\varepsilon\leq\lambda/4$, then the extra f\/ield $\phi$ vanishes.
\end{proof}

\begin{proof}[Proof Theorem \ref{T1.1}]
From Proposition~\ref{P4.1}, the complex Yang--Mills equations reduce to the pure Yang--Mills equation $\dd_{A}^{\ast}F_{A}=0$ and the curvature obeys $\|F_{A}\|_{L^{p}(X)}\leq\varepsilon$, then by the energy gap of Yang--Mills connection, we obtain that the connection $A$ is f\/lat.
\end{proof}

\begin{proof}[Proof Corollary \ref{C1.2}]
For a smooth solution $(A,\phi)$ of complex f\/lat connection, from the identity $F_{A}=\phi\wedge\phi$ and we apply Theorem~\ref{F6} to obtain
\begin{gather*}\|F_{A}\|_{L^{p}(X)}\leq\|\phi\wedge\phi\|_{L^{p}(X)}\leq C\|\phi\|^{2}_{L^{2}(X)},\end{gather*}
where $C=C(g,n,p)$. We can choose $\|\phi\|_{L^{2}(X)}$ suf\/f\/iciently small such that $\|F_{A}\|_{L^{p}(X)}\leq\varepsilon$, where~$\varepsilon$ is the constant in Theorem~\ref{T1.1}. Then we can prove that $\phi$ vanishes and $A$ is a~f\/lat connection.

It is easy to see the map $(A,\phi)\mapsto\|\phi\|_{L^{2}(X)}$ is continuous, then the moduli space of complex f\/lat connections is non-connected.
\end{proof}

\subsection*{Acknowledgements}
I would like to thank Karen Uhlenbeck and Michael Gagliardo for helpful comments regarding their article \cite{Gagliardo/Uhlenbeck:2012} and Paul Feehan for helpful comments regarding his articles~\cite{Feehan2014.09,Feehan2014.12,Feehan2015}. I thank the anonymous referees for a careful reading of my article and helpful comments and corrections. This work is partially supported by Wu Wen-Tsun Key Laboratory of Mathematics of Chinese Academy of Sciences at USTC.

\pdfbookmark[1]{References}{ref}
\LastPageEnding

\end{document}